\documentclass[12pt]{amsart}
\usepackage[colorlinks=true,pagebackref,hyperindex,citecolor=blue,linkcolor=red]{hyperref}
\usepackage{amsmath}
\usepackage{amsfonts}
\usepackage{setspace}
\usepackage{moreverb}
\usepackage[dvips]{graphicx}
\usepackage[latin1]{inputenc}
\usepackage[T1]{fontenc}
\usepackage{amssymb}
\usepackage{tikz}
\usepackage{color}
\usepackage[all]{xy}
\usepackage{xfrac}
\usepackage[top=1in, bottom=1in, left=1in, right=1in]{geometry}
\usepackage{mathrsfs}

\newtheorem{theoremx}{Theorem}

\newtheorem{theorem}{Theorem}[section]

\newtheorem{corollary}[theorem]{Corollary}
\newtheorem{lemma}[theorem]{Lemma}
\newtheorem{proposition}[theorem]{Proposition}

\theoremstyle{definition}
\newtheorem{definition}[theorem]{Definition}

\newtheorem{notation}[theorem]{Notation}

\newtheorem{remark}[theorem]{Remark}
\numberwithin{equation}{subsection}

\newcommand{\fpt}{{\operatorname{fpt}}}
\newcommand{\ct}{{\operatorname{ct}}}

\newcommand{\NN}{\mathbb{N}}
\newcommand{\ZZ}{\mathbb{Z}}

\newcommand{\m}{\mathfrak{m}}

\newcommand{\fa}{\mathfrak{a}}

\newcommand{\cA}{\mathcal{A}}

\newcommand{\Spec}{\operatorname{Spec}}

\newcommand{\Hom}{\operatorname{Hom}}

\newcommand{\Supp}{\operatorname{Supp}}

\newcommand{\Ht}{\operatorname{ht}}

\newcommand{\reg}{\operatorname{reg}}

\newcommand{\fp}{\mathfrak{p}}

\newcommand{\fq}{\mathfrak{q}}

\newcount\ancho \newcount\anchom \newcount\anchoa
\newcount\anchob \newcount\altura
\newcommand{\lhat}[3][0]{\altura=0 \advance\altura by #1
\ancho=#2 \anchom=\ancho \divide\anchom by 2
\anchoa=\ancho \divide\anchoa by 4
\anchob=\anchom \advance\anchob by \anchoa
\kern-3pt \begin{array}[b]{c}
\begin{picture}(1,1)(\anchom,-\altura)
\qbezier(0,2)(\anchoa,4)(\anchom,6)
\qbezier(\anchom,6)(\anchob,4)(\ancho,2)
\qbezier(0,2)(\anchoa,3.8)(\anchom,5.6)
\qbezier(\anchom,5.6)(\anchob,3.8)(\ancho,2)
\end{picture} \\[-4pt] {#3}
\end{array} \kern-4pt }

\begin{document}
\newcommand{\tens}{\otimes}
\newcommand{\hhtest}[1]{\tau ( #1 )}
\renewcommand{\hom}[3]{\operatorname{Hom}_{#1} ( #2, #3 )}

\title{$F$-Invariants of Stanley-Reisner Rings}
\author[W\'agner Badilla-C\'espedes]{W\'agner Badilla-C\'espedes$^1$}
\address{Centro de Investigaci\'on en Matem\'aticas, A.C., Jalisco S/N, Col. Mineral de Valenciana, C.P. $36023$, Guanajuato, GTO, M\'exico}  \email{wagner.badilla@cimat.mx}
\thanks{$^1$ Partially supported by CONACYT Grants $295631$ and $284598$, and partially supported through Luis N\'u\~nez-Betancourt's C\'atedras Marcos Moshinsky Grant.}

\begin{abstract}
In prime characteristic there are important invariants that allow us to measure singularities. For certain cases, it is known that they are rational numbers. In this article, we show this property for Stanley-Reisner rings in several cases.
\end{abstract}


\maketitle

\setcounter{tocdepth}{1}
\tableofcontents
\section{Introduction}
Throughout this manuscript $R$ denotes a Noetherian ring of prime characteristic $p$. In characteristic zero, the log canonical threshold, $\mathrm{lct}(f)$, of a polynomial $f$ with coefficients in a field, is an important invariant in birational geometry \cite{BFS}. This number measures the singularities of $f$ near to zero. In positive characteristic, the $F$-pure threshold of an element $f \in R$, denoted $\fpt(f)$, was defined by Takagi and Watanabe \cite{TWFpure}. Roughly speaking, this measures the splitting order of $f$. It is defined by
\begin{align*}
\fpt(f)=\sup \Bigg\{ \frac{a}{p^{e}}\;|\;\mathrm{the \; inclusion}\; Rf^{\frac{a}{p^{e}}} \subseteq R^{1/p^e} \;\mathrm{is \; a \; split} \Bigg\}
\end{align*}
for $f \in R$.

The $F$-pure threshold is considered as analogous to the log canonical threshold, and they share similar properties \cite{TWFpure,MTWFR}. In particular, if $f$ is an element in $\mathbb{Z}[x_1, \ldots, x_n]$, then $\lim\limits_{p \rightarrow \infty} {\fpt(f\;\mathrm{mod}\;p)}=\mathrm{lct}(f)$ \cite{HaraYoshida,MTWFR}.   

In this work, we study a general form of the $F$-pure threshold  called the Cartier threshold. Given $\fa$ and $J$ ideals in $R$, the Cartier threshold of $\fa$ with respect to $J$ is defined as $\ct_J(\mathfrak{a})= \lim\limits_{e \rightarrow \infty} {\frac{b_{\mathfrak{a}}^J(p^e)}{p^e}}$, where $$b_{\mathfrak{a}}^J(p^e)= \max \{t \in \mathbb{N}\;|\; \mathfrak{a}^t \not \subseteq J_e\} \; \& \; J_e=\{f	\in R \; |\; \varphi(f^{1/p^e})\in J, \;\forall \varphi\in\Hom_R(R^{1/p^e}, R )\}.$$ These numbers have been studied in more depth in an upcoming work \cite{DSHNBW}. If we consider $(R,\mathfrak{m},K)$ a local ring or a standard graded $K$-algebra which is $F$-finite and $F$-pure, then $\ct_{\mathfrak{m}}(\mathfrak{a})=\fpt(\mathfrak{a})$. 

In this manuscript, we focus on Stanley-Reisner rings. The combinatorial nature of these rings has been useful to 
study their structures in prime characteristic. For instance, in this case one can describe their algebras of Frobenius and Cartier  operators \cite{AMBZ,BZ}.
In this work, we  show that the Cartier threshold of $\mathfrak{a}$ with respect to $J$ in Stanley-Reisner rings is a rational number when $J$ is a radical ideal. 

\begin{theoremx}[{see Theorem \ref{pro1series} and Corollary \ref{mainresult}}]\label{MainThmA}
Let $\mathfrak{a}, \;J$ be two ideals in a Stanley-Reisner ring $R$, such that $\mathfrak{a} \subseteq J$, and $J$ is a radical ideal. Then, the Cartier threshold of $\mathfrak{a}$ with respect to $J$ is a rational number.  
\end{theoremx} 

In order to obtain Theorem \ref{MainThmA}, we need to reduce the computation of $\ct_J(\mathfrak{a})$ to the case where $J$ is a monomial ideal. For this trick, we need to replace $R$ by the completion of a suitable localization. 
Then, the problem is reduced to the regular case by taking a quotient with  respect to the Cartier core (see Definition \ref{defcartiercore}).

We now recall the definition of the $F$-thresholds. They are numbers obtained by comparing ordinary powers versus Frobenius powers. These were introduced in regular rings by Musta{\c{t}}{\u{a}}, Takagi and Watanabe \cite{MTWFR}, and their existence, in the general case, was proved by De Stefani, N\'u\~nez-Betancourt and P\'erez \cite{DSNBP}. These are defined as $c^{J}(\mathfrak{a})=\lim\limits_{e \rightarrow \infty} {\frac{\nu^{J}_{\mathfrak{a}}(p^{e})}{p^{e}}}$, where $\nu^{J}_{\mathfrak{a}}(p^{e})=\max \{m \in \mathbb{N}\; |\; \mathfrak{a}^m \not \subseteq J^{[p^{e}]} \}$, and $\mathfrak{a},\;J \subseteq R$ are ideals.

A recent line of research consists in understanding under which conditions the set of $F$-thresholds is a discrete set of rational numbers. This was proved by Blickle, Musta{\c{t}}{\u{a}}, and Smith \cite{BMMMSK} for an $F$-finite regular ring. Although the $F$-threshold is a rational number in the regular case, this situation is unknown in general Noetherian rings. Trivedi \cite{TR} showed that, in general, the set of $F$-thresholds with respect to the maximal ideal in a local ring is not necessarily discrete. In this paper, we study the  rationality of $F$-thresholds for Stanley-Reisner rings. 

\begin{theoremx}[{see Theorem \ref{pro5}}]\label{MainThmB}
 Let $\mathfrak{a}, \;J$ be two ideals in a Stanley-Reisner ring $R$, such that $\mathfrak{a} \subseteq \sqrt{J}$, and $J$ is a monomial ideal. Then, the $F$-threshold of $\mathfrak{a}$ with respect to $J$ is a rational number. 
\end{theoremx}

The key idea to prove Theorem \ref{MainThmB} is to work modulo the minimal primes, which yields a regular ring. The result follows from comparing the $F$-thresholds of $R$ versus these quotients.
We point out that Theorem \ref{MainThmB} is a key component of the proof of Theorem \ref{MainThmA}.

The Castelnuovo-Mumford regularity is an invariant that measures the complexity of the free resolution of a standard graded $K$-algebra $(R,\m,K)$. The growth of $\reg(R/J^{[p^e]})$ 
has been intensively studied due to its relation to  discreteness of $F$-
jumping coefficients \cite{KatzmanZhang,KSSZ,ZhangRegFrob}, localization of tight closure \cite{KatzmanComplexityFrob,HunekeLC}, and existence of the generalized Hilbert-Kunz multiplicity \cite{DaoSmirnov,VHK}. 
We recall that  the Castelnuovo-Mumford regularity can be computed in terms of the  $a$-invariants introduced by Goto and Watanabe \cite{GW1}. 
In this manuscript, we provide  a formula for the limits of $\reg(R/J^{[p^e]})$.

\begin{theoremx}[{see Theorem \ref{thmregularity}}]\label{MainThmC}
Let $J$ be a homogeneous ideal in a Stanley-Reisner ring $R$. Then, 
\begin{align*}
\lim\limits_{e \rightarrow \infty} {\frac{\reg(R/J^{[p^e]})}{p^{e}}}&=\max_{\substack{1 \leq i \leq d\\ \alpha \in \mathcal{A}'}}\{a_i(S/(J_\alpha+J))+|\alpha|\},
\end{align*}
where $\mathcal{A}'=\{\alpha \in \mathbb{N}^n\;|\;0\leq \alpha_i\leq 1\;$for$\;i=1,\ldots,n\}$, $J_{\alpha}=(I:x^\alpha)$, and $d=\max\{\dim(S/(J_\alpha+J))\;|\;\alpha \in \cA'\}$. In particular, this limit is an integer number.  
\end{theoremx}

\section{Stanley-Reisner Rings}\label{SecSR}

Throughout this section we use the following notation.
\begin{notation} \label{notationSR}
We denote $S=K[x_1,  \ldots ,x_n]$ with $K$ an $F$-finite field of prime characteristic $p$. Let $I$ be a squarefree monomial ideal of $S$. Let $I=\bigcap_{i=1}^{l}\mathfrak{p}_i$ such that $\mathfrak{p}_i \not \subseteq \mathfrak{p}_j$ for $i \not = j$ and $\mathfrak{p}_1, \ldots,\mathfrak{p}_l$ are generated by variables. We take $R=S/I$.
\end{notation}

These rings have mild singularities, for instance, they are $F$-pure. They also have combinatorial structure given by simplicial complexes.

Suppose that $\fa \subseteq R$ is an ideal. We abuse notation and denote the inverse image of $\fa \subseteq R$ under the natural projection $S \longrightarrow S/I$ by $\fa \subseteq S$.

We now characterize the ring of $p$-th roots of $R$ in terms of ideal quotients.
\begin{proposition} \label{propo 2}
If $q=p^{e}$, where $e$ is a nonnegative integer, then 
\begin{align*}
R^{1/q}=S^{1/q} / I^{1/q} \cong \bigoplus_{\substack{1 \leq i \leq s\\ \alpha \in \mathcal{A}}}S/J_{\alpha}(a_ix^\alpha)^{1/q},
\end{align*}
with $\mathcal{A}=\{\alpha \in \mathbb{N}^n\;|\;0\leq \alpha_i\leq q-1\;$for$\;i=1,\ldots, n\}$, $\mathcal{B}=\{{a_i}^{1/q}\;|\;i=1,\ldots,s\}$ is a base of $K^{1/q}$ as $K$-vector space, and $J_{\alpha}=(I:x^\alpha)$.
\end{proposition}  
\begin{proof}
Each element $r^{1/q} \in S^{1/q}$ can be written uniquely as    
$$r^{1/q} = \bigoplus_{\substack{1 \leq i \leq s\\ \alpha \in \mathcal{A}}}r_{i,\alpha}(a_ix^\alpha)^{1/q},$$ where $r_{i,\alpha} \in S$. We take 
$$\varphi:S^{1/q}\longrightarrow \bigoplus_{\substack{1 \leq i \leq s\\ \alpha \in \mathcal{A}}}S/J_{\alpha}(a_ix^\alpha)^{1/q},$$ 
defined by 
$$\varphi(r^{1/q})=\bigoplus_{\substack{1 \leq i \leq s\\ \alpha \in \mathcal{A}}}(r_{i,\alpha}+J_{\alpha})(a_ix^\alpha)^{1/q}.$$
We have that $\varphi$ is a surjective $S$-linear morphism.

We claim that $\ker\varphi = I^{1/q}$. First, we show $\ker\varphi \subseteq I^{1/q}$. Let $r^{1/q}\in \ker\varphi$. It is sufficient to consider $r^{1/q}=x^{\theta}(a_ix^\alpha)^{1/q}$ for some $\theta \in \mathbb{N}^n, \;\alpha \in \mathcal{A}$, and $i \in \{1,\ldots,s\}$. Hence, $0=\varphi(r^{1/q})=(x^{\theta}+J_{\alpha})(a_ix^\alpha)^{1/q}$. Thus, $x^{\theta} \in J_{\alpha}$. This implies that $a_ix^{\alpha+\theta}\in I$, and so, $x^{\theta/q}(a_ix^\alpha)^{1/q} \in I^{1/q}$. It follows that $r^{1/q}=x^{\theta}(a_ix^\alpha)^{1/q}=(x^{\theta/q})^q(a_ix^\alpha)^{1/q} \in I^{1/q}$. 

Now, we show $I^{1/q} \subseteq \ker\varphi$. Let $r^{1/q}\in I^{1/q}$. It is enough to consider $r^{1/q}=x^{\theta}(a_ix^\alpha)^{1/q}x^{\beta/q}$ with $\theta \in \mathbb{N}^n, \;\alpha \in \mathcal{A},\;i \in\{1,\ldots,s\}$, and $x^{\beta}$ a generator of $I$. Since $0 \leq \alpha_j \leq q-1$ and $0 \leq \beta_j \leq 1$ for every $1 \leq j \leq n$, there exists $\gamma \in \{0,1\}^n$ 
such that $\alpha+ \beta - q\gamma \in \mathcal{A}$. Let $\alpha'=\alpha+ \beta - q\gamma$. We note that $x^{\theta+\gamma}(a_ix^{\alpha'}) \in I$. As a consequence, $x^{\theta+ \gamma} \in J_{\alpha'}$. Furthermore, $r^{1/q}=x^{\theta+ \gamma} (a_ix^{\alpha'})^{1/q}$. Subsequently, $\varphi(r^{1/q})=(x^{\theta+\gamma}+J_{\alpha'})(a_ix^{\alpha'})^{1/q}=0$. Thus, $r^{1/q} \in \ker  \varphi$.  

It follows,
\begin{align*}
R^{1/q} \cong \bigoplus_{\substack{1 \leq i \leq s\\ \alpha \in \mathcal{A}}}S/J_{\alpha}(a_ix^\alpha)^{1/q}
\end{align*} 
as $S$-module. Therefore, they are isomorphic as $R$-modules. 
\end{proof}

\begin{remark}\label{remarkdimension}
We follow Notation \ref{notationSR}. Let $\mathfrak{q}$ be a prime ideal of $S$. Suppose that $\mathfrak{p}_1, \ldots, \mathfrak{p}_r \subseteq \mathfrak{q}$ with $r \leq l$, $\mathfrak{p}_j \not \subseteq \mathfrak{q}$ for $r<j$, and $(x_1, \ldots, x_u)=\sum_{i=1}^{r}\mathfrak{p}_i$.

Let $\widetilde{\mathfrak{q}}_0,\ldots, \widetilde{\mathfrak{q}}_t \in \Spec S_\mathfrak{q}$ be such that $(x_1,\ldots,x_u)S_\mathfrak{q}=\widetilde{\mathfrak{q}}_0 \varsubsetneqq \widetilde{\mathfrak{q}}_1   \varsubsetneqq \ldots \varsubsetneqq \widetilde{\mathfrak{q}}_t$. There exist $\mathfrak{q}_0,  \ldots,  \mathfrak{q}_t \in \Spec S$, where $\mathfrak{q}_i \subseteq \mathfrak{q}$ and $\mathfrak{q}_i=\widetilde{\mathfrak{q}}_i \cap S$. We have that
\begin{align*}
(0)\varsubsetneqq (x_1)\varsubsetneqq (x_1,x_2)\varsubsetneqq   \ldots\varsubsetneqq (x_1, \ldots,x_u)=\mathfrak{q}_0 \varsubsetneqq \mathfrak{q}_1 \varsubsetneqq \ldots \varsubsetneqq \mathfrak{q}_t \subseteq \mathfrak{q}
\end{align*}
is a chain of prime ideals in $S$ with length $u+t$, and so, $u+t \leq \Ht(\mathfrak{q})$. Hence, $t \leq \Ht(\mathfrak{q})-u$. Then, $\dim {S_\mathfrak{q}/(x_1, \ldots,x_u)S_\mathfrak{q}} \leq \Ht(\mathfrak{q})-u$. Therefore, 
$$\Ht(\mathfrak{q})-u=\dim {S_\mathfrak{q}/(x_1, \ldots,x_u)S_\mathfrak{q}}.$$

In particular, if we take $A=\widehat{S_\mathfrak{q}}$, we have that 
$$\dim A-u=\dim {A/(x_1, \ldots, x_u)A}.$$
Since $A$ is a complete regular local ring, $A \cong L[[x_1,\ldots,x_u,y_1,\ldots ,y_t]]$, where $K \subseteq L$ is a field extension. Moreover, we have that $IA=\bigcap_{i=1}^{l}\mathfrak{p}_iA$ is squarefree monomial ideal of $A$ in variables $x_1, \ldots, x_u$. We denote $\underline{x}^\theta={x_1}^{\theta_1}\cdots {x_u}^{\theta_u}{y_1}^{\theta_{u+1}}\cdots {y_t}^{\theta_{u+t}}$. We take $B=A/IA$ and $\mathfrak{m}$ its maximal ideal.  
\end{remark} 

\begin{proposition} \label{propo4} 
If $q=p^{e}$, where $e$ is a nonnegative integer, then 
\begin{align*}
B^{1/q} \cong \bigoplus_{\substack{1 \leq i \leq s\\ \alpha \in \mathcal{A} }}A/J_{\alpha}(a_i\underline{x}^\alpha)^{1/q},
\end{align*}
with $ \mathcal{A}=\{\alpha \in \mathbb{N}^{u+t}\;|\;0\leq \alpha_i\leq q-1\;$for$\;i=1, \ldots, u+t\}$, $\mathcal{B}=\{{a_i}^{1/q}\;|\;i=1, \ldots,s\}$ is a base of $L^{1/q}$ as $L$-vector space, and $J_{\alpha}=(IA:\underline{x}^\alpha)$.
\end{proposition} 
\begin{proof}
The proof is analogous to Proposition \ref{propo 2}.
\end{proof}

\section{$F$-Thresholds}

The $F$-thresholds were introduced by Musta{\c{t}}{\u{a}}, Takagi and Watanabe \cite{MTWFR} for $F$-finite regular local rings of prime characteristic. Subsequently, in work with Huneke \cite{HMTWF-thresholdsgeneralcase}, they defined $F$-thresholds in general rings of positive characteristic, provided the limit exists. The existence of these invariants in the general case is proved in the work of De Stefani, N\'u\~nez-Betancourt and P\'erez \cite{DSNBP}.

Our main goal is to describe the $F$-thresholds of Stanley-Reisner rings, when we have monomial ideals.


\subsection{Definition and First Properties}
In this subsection $R$ denotes a ring of prime characteristic $p$. We discuss properties related to $F$-thresholds.
\begin{definition} Let $R$ be a ring. Given $\mathfrak{a}, \; J$ ideals inside $R$ such that $\mathfrak{a} \subseteq \sqrt{J}$, we define 
\begin{align*}
\nu^{J}_{\mathfrak{a}}(p^{e})&=\max \{m \in \mathbb{N}\; |\; \mathfrak{a}^m \not \subseteq J^{[p^{e}]} \}.
\end{align*}
\end{definition}

\begin{lemma}[{\cite[Lemma $3.3$]{DSNBP}}]\label{lem4} 
Let $R$ be a ring, and $\mathfrak{a}, \; J$ ideals of $R$ such that $\mathfrak{a} \subseteq \sqrt{J}$. Then,  
\begin{align*}
\frac{\nu^{J}_{\mathfrak{a}}(p^{e_1+e_2})}{p^{e_1+e_2}}-\frac{\nu^{J}_{\mathfrak{a}}(p^{e_1})}{p^{e_1}} \leq \frac{\mu(\mathfrak{a})}{p^{e_1}}
\end{align*}
for every $e_1, \;e_2 \in \mathbb{N}.$
\end{lemma}  

\begin{theorem}[{\cite[Theorem $3.4$]{DSNBP}}]  
Let $R$ be a ring, and $\mathfrak{a}, \; J$ be two ideals in $R$ such that $\mathfrak{a} \subseteq \sqrt{J}$. Then, $\lim\limits_{e \rightarrow \infty} {\frac{\nu^{J}_{\mathfrak{a}}(p^{e})}{p^{e}}} $ exists.
\end{theorem}   
The previous theorem gives existence to the $F$-thresholds and we may define them.
\begin{definition} [{\cite{DSNBP}}]
Let $R$ be a ring. Given $\mathfrak{a}, \; J$ ideals of $R$ such that $\mathfrak{a} \subseteq \sqrt{J}$, we define the $F$-threshold of $\mathfrak{a}$ with respect to $J$ by 
\begin{align*}
c^{J}(\mathfrak{a})&=\lim\limits_{e \rightarrow \infty} {\frac{\nu^{J}_{\mathfrak{a}}(p^{e})}{p^{e}}}.
\end{align*}
\end{definition}
\begin{proposition} [{\cite[Proposition $2.7$]{MTWFR} $\&$ \cite[Proposition $2.2$]{HMTWF-thresholdsgeneralcase}}]\label{cpFthreshold} 
Let $R$ be a ring, and let $\mathfrak{a},\;I,\;J$ be ideals in $R$. Then, the following hold.
\begin{itemize}
\item[(1)]If $J \subseteq I$, and $\mathfrak{a} \subseteq \sqrt{J}$, then $c^{I}(\mathfrak{a})\leq c^{J}(\mathfrak{a})$.
\item[(2)] If $\mathfrak{a} \subseteq \sqrt{J}$, then $c^{J^{[p]}}(\mathfrak{a})=p \cdot c^{J}(\mathfrak{a})$. 
\end{itemize}
\end{proposition} 


\subsection{$F$-Thresholds in Stanley-Reisner Rings}

In this subsection, we focus on Stanley-Reisner rings. We denote $S=K[x_1, \ldots ,x_n]$ with $K$ an $F$-finite field of prime characteristic $p$. Let $I$ be a squarefree monomial ideal of $S$, and $R=S/I$.

Suppose that $\fa \subseteq R$ is an ideal. We abuse notation and denote the inverse image of $\fa \subseteq R$ under the natural projection $S \longrightarrow S/I$ by $\fa \subseteq S$.

The following proposition is one of the main results of this paper, Theorem \ref{MainThmB}. Using the fact that the quotient of $R$ with each of its minimal prime ideals is a regular ring, we obtain a case where the $F$-threshold is a rational number. 

\begin{theorem} \label{pro5}  
Let $\mathfrak{a},\;J$ be ideals of $R$, with $\mathfrak{a} \subseteq \sqrt{J}$, and $J$ monomial. Let $\mathfrak{p}_1,\; \ldots,\;\mathfrak{p}_l$ be the minimal prime ideals of $R$. Then, 
\begin{align*}
c^{J}_{R}(\mathfrak{a})&= \max \left\{c^{JS/\fp_i}_{S/{\mathfrak{p}_i}}(\mathfrak{a}S/\fp_i)\right\}. 
\end{align*} 
In particular, $c^{J}_{R}(\mathfrak{a}) \in \mathbb{Q}$. 
\end{theorem} 
\begin{proof}
We know that $I=\bigcap_{i=1}^{l}\mathfrak{p}_i$. Moreover, each $\mathfrak{p}_i$ is generated by variables. We claim that $c^{J}_{R}(\mathfrak{a})\geq \max \left\{c^{JS/\fp_i}_{S/{\mathfrak{p}_i}}(\mathfrak{a}S/\fp_i)\right\}$. Let $e$ be a nonnegative integer. We take $t_i=\nu_{\fa S/{\mathfrak{p}_i}}^{JS/\fp_i}(p^{e})$. Then, $\mathfrak{a}^{t_i}S/{\mathfrak{p}_i} \not \subseteq J^{[p^{e}]}S/{\mathfrak{p}_i}$. Hence, there exists $r \in \mathfrak{a}^{t_i}$ such that $r-c \not \in \mathfrak{p}_i$ for every $c \in J^{[p^{e}]}$. Thus, $r-c \not \in I$, and so $r \not \in  J^{[p^{e}]}$. As a consequence, $\mathfrak{a}^{t_i} \not \subseteq J^{[p^{e}]}$.

We have that $t_i \leq \nu_{\fa}^{J}(p^{e})$ for all $i$. Then, $\frac{\nu_{\fa S/{\mathfrak{p}_i}}^{JS/\fp_i}(p^{e})}{p^{e}} \leq \frac{\nu_{\fa}^{J}(p^{e})}{p^{e}}$. Thus, $c^{JS/\fp_i}_{S/{\mathfrak{p}_i}}(\mathfrak{a}S/\fp_i) \leq c^{J}_{R}(\mathfrak{a})$. Therefore,  $c^{J}_{R}(\mathfrak{a})\geq \max \left\{c^{JS/\fp_i}_{S/{\mathfrak{p}_i}}(\mathfrak{a}S/\fp_i)\right\}$.

We now show that $\bigcap_{i=1}^{l}(J^{[p^{e}]}+\mathfrak{p}_i) \subseteq J^{[p^{e}]}+I$. We proceed by contradiction. Let $s$ be a generator of $\bigcap_{i=1}^{l}(J^{[p^{e}]}+\mathfrak{p}_i)$ such that $s \not \in J^{[p^{e}]}+I$. Since $ J^{[p^{e}]}$ and each $\mathfrak{p}_i$ are monomial ideals, we have that every $J^{[p^{e}]}+\mathfrak{p}_i$ is a monomial ideal too. Hence, $\bigcap_{i=1}^{l}(J^{[p^{e}]}+\mathfrak{p}_i)$ is a monomial ideal. We can take $s$ as a monomial. Furthermore, $s \not \in J^{[p^{e}]}$ and $s \not \in I$. Thus, there exists an $i$ such that $s \not \in \mathfrak{p}_i$. However, $s \in J^{[p^{e}]}+\mathfrak{p}_i$. Since $s$ is a monomial and $\mathfrak{p}_i$ is generated by variables, we conclude that $s \in J^{[p^{e}]}$, we get a contradiction. Thus, $s \in J^{[p^{e}]}+I$.  

We prove that $c^{J}_{R}(\mathfrak{a})\leq \max \left\{c^{JS/\fp_i}_{S/{\mathfrak{p}_i}}(\mathfrak{a}S/\fp_i)\right\}$.  Let $e$ be a nonnegative integer. We take $t=\nu_{\fa}^{J}(p^{e})$. Then, $\mathfrak{a}^t \not \subseteq J^{[p^{e}]}$. Hence, there exists $r \in \mathfrak{a}^t$ such that $r-c \not \in I$ for every $c \in J^{[p^{e}]}$. As a consequence, $r \not \in  J^{[p^{e}]} + I$, and so $r \not \in \bigcap_{i=1}^{l}(J^{[p^{e}]}+\mathfrak{p}_i)$. Hence, $r \not \in J^{[p^{e}]} + \mathfrak{p}_i$ for some $i$. It follows that $\mathfrak{a}^tS/\mathfrak{p}_i \not \subseteq J^{[p^{e}]}S/\mathfrak{p}_i$.  

Consequently, we have $t \leq \nu_{\fa S/{\mathfrak{p}_i}}^{JS/\fp_i}(p^{e})$ for some $i$. Then, $\frac{\nu_{\fa}^{J}(p^{e})}{p^{e}} \leq \max \left\{ \frac{\nu_{\fa S/{\mathfrak{p}_i}}^{JS/\fp_i}(p^{e})}{p^{e}}\right\}$. Therefore,  $c^{J}_{R}(\mathfrak{a})\leq \max \left\{c^{JS/\fp_i}_{S/{\mathfrak{p}_i}}(\mathfrak{a}S/\fp_i)\right\}$.
\end{proof}

\begin{remark}\label{remarkseries}
Given $\widetilde{S}=K[[x_1,  \ldots ,x_n]]$ with $K$ an $F$-finite field of prime characteristic $p$. We take $\widetilde{I}$ as a squarefree monomial ideal of $\widetilde{S}$, and $\widetilde{R}=\widetilde{S}/\widetilde{I}$, same as in Theorem \ref{pro5}. Let $\widetilde{\mathfrak{a}},\;\widetilde{J}$ be two ideals of $\widetilde{R}$, with $\widetilde{\mathfrak{a}} \subseteq \sqrt{\widetilde{J}}$, and $\widetilde{J}$ monomial. Then, $c^{\widetilde{J}}_{\widetilde{R}}(\widetilde{\mathfrak{a}}) \in \mathbb{Q}$.  
\end{remark}

\section{The Ideal $J_e$}

In this section we introduce the Cartier core of an ideal. This is related to the Cartier operators. We also investigate the behavior of these ideals for Stanley-Reisner rings. 

\subsection{Cartier Contraction.}
We begin this subsection introducing an ideal that allows the study of homomorphisms that do not give splittings. 

\begin{definition}[{\cite{AEFpure}}]
 Let $(R,\mathfrak{m},K)$ be a local ring or a standard graded $K$-algebra, which is $F$-finite and $F$-pure. We define 
 \begin{align*}
 I_e(R)&=\{f \in R\;|\; \varphi(f^{1/p^e}) \in \mathfrak{m},\; \mathrm{for\; all}\; \varphi \in \Hom_R(R^{1/p^e},R)\},
\end{align*}
where $e\in \mathbb{N}$.  
\end{definition}
\begin{remark}
The set $I_e(R)$ is an ideal of $R$, and is called the $e$-th splitting ideal of $R$. Then, $f \not \in I_e(R)$ if and only if $\varphi(f^{1/p^e})=1$ for some map $ \varphi \in \Hom_R(R^{1/p^e},R)$. 
\end{remark}

The ideal $I_e(R)$ is used to define the $F$-signature \cite{YAO}. Smith and Van den Bergh in their work \cite{KSVB} showed existence of this invariant when the ring $R$ is strongly $F$-regular and has finite Frobenius representation type. Later, Huneke and Leuschke \cite{CHLG} showed that this invariant exists if $R$ is a complete local Gorenstein domain. 
For rings that are Gorenstein on the punctured spectrum, its existence was given by Yao \cite{YAO}. Subsequently, Tucker \cite{TKFSIG} showed existence of the $F$-signature in $R$ with full generality.

\begin{definition}[{\cite{DSHNBW}}] \label{defJ_e}
 Let $R$ be an $F$-finite $F$-pure ring, and $J$ be an ideal in $R$. We define the Cartier contraction as 
 \begin{align*}
 J_e&=\{f \in R\;|\; \varphi(f^{1/p^e}) \in J,\; \mathrm{for\; all}\; \varphi \in \Hom_R(R^{1/p^e},R)\},
\end{align*}
for $e\in \mathbb{N}$.  
\end{definition}

\begin{remark}
The set $J_e$ is an ideal of $R$. If $(R,\mathfrak{m},K)$ is a local ring or a standard graded $K$-algebra, and $\mathfrak{m}=J$, we have $I_e(R)=J_e$. 
\end{remark}

\begin{proposition}\label{containJ_e}
Let $R$ be an $F$-finite $F$-pure ring, and $J$ be an ideal of $R$. Then, for every $e$ nonnegative integer $J^{[p^e]} \subseteq J_e \subseteq J$.
\end{proposition}
\begin{proof}
First, we show the inclusion $J^{[p^e]} \subseteq J_e$. Let $x$ be an element of $J$. For every $ \varphi \in \Hom_R(R^{1/p^e},R)$, $\varphi((x^{p^e})^{1/p^e})=\varphi(x \cdot 1)=x\varphi(1)\in J$. Therefore, $x^{p^e} \in J_e$.

To show the other inclusion, we proceed by contrapositive. Let $r \not \in J$. Since $R \subseteq R^{1/{p^e}}$ is an $R$-module split, we can take $\beta \in \Hom_R(R^{1/p^e},R)$ such that $\beta|_R=1_R$. It follows that $\beta((r^{p^e})^{1/{p^e}})=\beta(r)=r \not \in J$. Hence, $r^{p^e} \not \in J_e$, and so, $r \not \in J_e$.
\end{proof}

The equality $J_e=J$ holds under certain conditions. This is done in Proposition \ref{pro2} below. 

The following proposition shows that the construction of the ideals $J_e$ commutes with arbitrary intersections.

\begin{proposition} \label{pro6} 
Let $R$ be an $F$-finite $F$-pure ring, and $\{J_i\}_i$ be a family of ideals in $R$. Then, $\left(\bigcap_iJ_i\right)_e=\bigcap_i(J_i)_e$ for every $e$ nonnegative integer.
\end{proposition}
\begin{proof}
For every $\varphi \in \Hom_R(R^{1/p^e},R)$, we have that
\begin{align*}
x \in\left(\bigcap_iJ_i\right)_e& \Leftrightarrow \varphi(x^{1/p^e}) \in \bigcap_iJ_i\\
&\Leftrightarrow \varphi(x^{1/p^e}) \in J_i \; \mathrm{for \; every} \; i  \\
&\Leftrightarrow x \in (J_i)_e \; \mathrm{for \; every} \; i\\
&\Leftrightarrow x \in\bigcap_i(J_i)_e. 
\end{align*}
\end{proof}

\begin{proposition} \label{primary}
Let $R$ be an $F$-finite $F$-pure ring, and $\fq$ be a prime ideal of $R$. Then $\fq_e$ is a $\fq$-primary ideal for every $e \in \mathbb{N}$. 
\end{proposition}
\begin{proof}
We show that $\sqrt{\fq_e}=\fq$. By Proposition \ref{containJ_e}, $\fq^{[p^e]} \subseteq \fq_e \subseteq \fq$, and so, $$\fq= \sqrt{\fq}=\sqrt{\fq^{[p^e]}}\subseteq \sqrt{\fq_e} \subseteq \sqrt{\fq}=\fq.$$

We now show that $\fq_e$ is primary. Suppose that there exist $a,b \in R$ such that $a \not \in \fq_e$ and $b \not \in \fq$. There is $\varphi \in \Hom_R(R^{1/p^e},R)$ satisfying $\varphi(a^{1/p^e}) \not \in \fq$. As $\fq$ is a prime ideal, $\varphi((b^{p^e}a)^{1/p^e})=\varphi(ba^{1/p^e})=b\varphi(a^{1/p^e}) \not \in \fq$. Hence, $b^{p^e}a \not \in \fq_e$, and so, $ab \not \in \fq_e$. Therefore, $\fq_e$ is a $\fq$-primary ideal of R.  
\end{proof}

We now recall the definition of uniformly compatible. Our goal is to study the biggest uniformly compatible ideal contained in other given ideal.

\begin{definition}[{\cite{KSCenterofF-purity}}]\label{defF-compatible}
Let $R$ be an $F$-finite ring, and $J$ be an ideal of $R$. We say that $J$ is uniformly $F$-compatible if $\varphi(J^{1/p^{e}})\subseteq J$ for every $e>0$ and every $\varphi \in \Hom_R(R^{1/p^e},R)$.
\end{definition}

\begin{proposition} \label{pro2} 
Let $R$ be an $F$-finite $F$-pure ring. Let $J$ be an ideal of $R$. Then, $J_e = J$ for every $e$ nonnegative integer if and only if $J$ is uniformly $F$-compatible.  
\end{proposition}
\begin{proof}
We suppose that  $J_e = J$ for every $e \geq 0$. We have that $\varphi(J^{1/p^e}) \subseteq J$ for every $ \varphi \in \Hom_R(R^{1/p^e},R)$ by Definition \ref{defJ_e}.

For the other direction, it is enough to see that $J \subseteq J_e$ for every $e > 0$. In fact, by Definition \ref{defF-compatible}, $\varphi(J^{1/p^e}) \subseteq J$ for all $ \varphi \in \Hom_R(R^{1/p^e},R)$. Therefore, $J \subseteq J_e$.           
\end{proof}

\begin{lemma} \label{lema6}
Let $R$ be an $F$-finite $F$-pure ring. Let $J$ be an ideal of $R$. Then,
$\bigcap_{s \in \mathbb{N}} J_s$ is uniformly $F$-compatible.   
\end{lemma}
\begin{proof}
We proceed by contradiction. We suppose that $\varphi\left( \left( \bigcap_{s \in \mathbb{N}} J_s \right)^{1/p^e}\right) \not \subseteq \bigcap_{s \in \mathbb{N}} J_s$ for some $e>0$ and $\varphi \in \Hom_R(R^{1/p^e},R)$, and so, we have an $f \in \bigcap_{s \in \mathbb{N}} J_s$ such that  $\varphi(f^{1/p^e}) \not \in \bigcap_{s \in \mathbb{N}} J_s$. Thus, $\varphi(f^{1/p^e}) \not \in J_s$ for some $s \in \mathbb{N}$. Consequently, there exists $\phi \in \Hom_R(R^{1/p^s},R)$ such that $\phi(\varphi(f^{1/p^e})^{1/p^s}) \not \in J$.

If we take $\psi:R^{1/p^{e+s}} \longrightarrow R^{1/p^s}$ such that $\psi(r^{1/p^{e+s}})=\varphi(r^{1/p^e})^{1/p^s}$, we have that $\psi$ is $R$-linear. As a consequence, $\sigma=\phi \circ \psi \in \Hom_R(R^{1/p^{e+s}},R)$. Then, $$\sigma(f^{1/p^{e+s}})= \phi \circ \psi (f^{1/p^{e+s}})= \phi( \varphi(f^{1/p^e})^{1/p^s}) \not \in J.$$
Therefore, $f \not \in J_{e+s}$, and we reach a contradiction.      
\end{proof}

\begin{proposition}
Let $R$ be an $F$-finite $F$-pure ring. Let $J$ be an ideal of $R$. Then, $\bigcap_{s \in \mathbb{N}} J_s$ is the biggest uniformly $F$-compatible ideal contained in $J$.
\end{proposition}
\begin{proof}
Let $I \subseteq J$ be an uniformly $F$-compatible ideal. By Proposition \ref{pro2}, $I=I_e \subseteq J_e$ for every $e \geq0$. Therefore, $I \subseteq \bigcap_{s \in \mathbb{N}} J_s$.
\end{proof}

Motivated by the splitting prime ideal \cite{AEFpure} and differential core \cite{BJLdifferentialoperators}, we introduce the Cartier core.
\begin{definition}\label{defcartiercore}
Let $R$ be an $F$-finite $F$-pure ring. Given $J$ an ideal of $R$, we define the Cartier core of $J$ as $\mathcal{P}(J)=\bigcap_{s \in \mathbb{N}} J_s$.     
\end{definition}

\begin{remark}
Let $(R,\mathfrak{m},K)$ be a local ring or a standard graded $K$-algebra, and $\mathfrak{m}=J$. Then, the ideal $\mathcal{P}(J)$ coincides with the splitting prime of $R$, denoted $\mathcal{P}(R)$, and introduced by Aberbach and Enescu \cite{AEFpure}.   
\end{remark}

In Proposition \ref{tore5}, we see a characterization of the Cartier core. This plays an important role in  Subsection \ref{subsec1} to describe the ideal $\fq_e$ for Stanley-Reisner rings. 

\begin{remark} \label{remarkCartierRadical}
Let $R$ be an $F$-finite $F$-pure ring, and $J$ be an ideal of $R$. For every $r \in \sqrt{\mathcal{P}(J)}$, $r^{p^e} \in \mathcal{P}(J)$ for some $e\in \mathbb{N}$. Since $R \subseteq R^{1/{p^e}}$ is an $R$-module split, there exists $\beta \in \Hom_R(R^{1/p^e},R)$ such that $\beta|_R=1_R$. Moreover, $r=(r^{p^e})^{1/{p^e}} \in \left( \mathcal{P}(J)\right) ^{1/p^e}$. Thus, $r= \beta(r) \in \mathcal{P}(J)$ by Lemma \ref{lema6}. Therefore, the Cartier core of $J$ is a radical ideal.
\end{remark}

Since $J_{s+1}$ is not necessarily contained in $J_s$, we need to show that $\bigcap_{s \geq e} J_s$  is the Cartier core for any $e$. 
  
\begin{proposition} \label{tore5}
Let $R$ be an $F$-finite $F$-pure ring, and $J$ be an ideal of $R$. Then, $\mathcal{P}(J) = \bigcap_{s \geq e} J_s$ for every nonnegative integer $e$.
\end{proposition}
\begin{proof}
We must show that $\bigcap_{s \geq e} J_s \subseteq \mathcal{P}(J)$. Let $x\in \bigcap_{s \geq e} J_s$. Thus $x \in J$ by Proposition \ref{containJ_e}. Hence, $x^{p^s} \in J^{[p^s]}$ for every $s \leq e$. As a consequence, $x^{p^e} \in J^{[p^s]}$. As  $x^{p^e} \in \bigcap_{s \geq e} J_s$, we have that $x^{p^e} \in \mathcal{P}(J)$. Thus, $x\in \sqrt{\mathcal{P}(J)}$. Therefore, $x \in \mathcal{P}(J)$ by Remark \ref{remarkCartierRadical}.       
\end{proof}


\subsection{The Ideal $\fq_e$ in Stanley-Reisner Rings} \label{subsec1}

Throughout this subsection, we denote $S=K[x_1,  \ldots ,x_n]$ with $K$ an $F$-finite field of prime characteristic $p$. Let $I$ be a squarefree monomial ideal of $S$, $R=S/I$, and $\fp_1, \ldots,\fp_l$ are the minimal prime ideals of $R$. We want to compute the ideal $\fq_e$, when $\fq$ is a monomial prime ideal of $R$.

\begin{lemma}\label{monomial}
Let $J$ be a monomial ideal in $R$, and $e$ be a nonnegative integer. Then, $J_e$ and $\mathcal{P}(J)$ are monomial ideals.
\end{lemma}
\begin{proof}
We note that $R$ is $\NN^n$-graded, and $R^{1/p^{e}}$ is $\frac{1}{p^{e}}\NN^n$-graded $R$-module. As $R \subseteq R^{1/p^{e}}$, we can view $R$ as an $\frac{1}{p^{e}}\NN^n$-graded $R$-module. To show that $J_e$ is a monomial ideal, it suffices to prove that $J_e$ is a homogeneous ideal with the $\NN^n$ grading. Let $r=r_{\alpha_1}+\cdots+r_{\alpha_t} \in J_e$, with $r_{\alpha_i} \in R$ of degree $\alpha_i \in \NN^{n}$. Let $\varphi \in \Hom_R(R^{1/p^{e}},R)$. Since $R^{1/p^{e}}$ is a finitely generated $R$-module, every homomorphism $R^{1/p^{e}} \longrightarrow R$ is a finite sum of graded homomorphisms. Thus, we can take $\varphi$ homogeneous of degree $\beta \in \frac{1}{p^{e}}\NN^n$. Then, $$\varphi(r^{1/p^{e}})=\varphi(r_{\alpha_1}^{1/p^{e}})+ \cdots + \varphi(r_{\alpha_t}^{1/p^{e}}) \in J,$$ and each $\varphi(r_{\alpha_i}^{1/p^{e}})$ has degree $\frac{1}{p^{e}}\alpha_i + \beta$. As $J$ is a homogeneous ideal, $\varphi(r_{\alpha_i}^{1/p^{e}})\in J$. Then, $r_{\alpha_i} \in J_e$ for all $i \in \{1,\ldots,t\}$. Therefore, $J_e$ is a homogeneous ideal.

Since $J_e$ is a monomial ideal, $\mathcal{P}(J)$ is a monomial ideal because intersection of monomial ideals is monomial.
\end{proof}

\begin{proposition}\label{proQ_e}
Given $\fq$ a monomial prime ideal of $R$, then $\fq_e=\fq^{[q]}+\mathcal{P}(\fq)$ for every $e \in \mathbb{N}$, and $q=p^e$.
\end{proposition}
\begin{proof}
We must show $\fq_e \subseteq \fq^{[q]}+\mathcal{P}(\fq)$. We proceed by contradiction. Let $r$ be an element in $\fq_e$. We suppose that $r \not \in \fq^{[q]}+\mathcal{P}(\fq)$. From Lemma \ref{monomial}, $\fq_e$ is a monomial ideal of $R$. Then, we can take $r=x^{\beta}$, with $\beta \in \mathbb{N}^n$.

Thus, $x^{\beta} \not \in \fq^{[q]}$, and $x^{\beta} \not \in \mathcal{P}(\fq)$. By Proposition \ref{tore5}, $x^{\beta} \not \in \bigcap_{s \geq e} \fq_s$, and so, there exists $e' \geq e$ such that $x^{\beta} \not \in \fq_{e'}$.

Let $\mathcal{A}=\{\alpha \in \mathbb{N}^n\;|\;0\leq \alpha_i\leq q-1\;$for$\;i=1,\ldots,n\},\;\mathcal{A'}=\{\alpha' \in \mathbb{N}^n\;|\;0\leq \alpha'_i\leq p^{e'}-1\;$for$\;i=1,\ldots,n\}$, $\mathcal{B}=\{{a_i}^{1/q}\;|\;i=1,\ldots,s\}$ be a base of $K^{1/q}$ as $K$-vector space, and $\mathcal{B}'=\{(a'_i)^{1/p^{e'}}\;|\;i=1,\ldots,s'\}$ be a base of $K^{1/p^{e'}}$ as $K$-vector space. We may suppose that $a_1=a'_1=1$. 

Moreover, $x^{\beta/p^e}=x^\theta x^{\alpha/p^e}$ and $x^{\beta/p^{e'}}=x^{{\theta}'} x^{{\alpha}'/p^{e'}}$, with $\theta,\;{\theta}' \in \mathbb{N}^n$, $\alpha \in \mathcal{A}$, and ${\alpha}' \in \mathcal{A'}$. As $p^{e'} \geq p^e$, then 
$\alpha_i \leq {\alpha}'_i$ and $\theta_i \geq {\theta}'_i$ for every $i$. Thus, there exists $\tau_i \in \mathbb{N}$ such that $\theta_i = {\theta}'_i+\tau_i$.

Furthermore, $J_{\alpha}=(I:x^{\alpha}) \subseteq (I:x^{{\alpha}'})=J_{{\alpha}'} $. Hence, we take a morphism $$\phi \in \Hom_R((S/J_{\alpha}) x^{\alpha/p^e},(S/J_{{\alpha}'}) x^{{\alpha}'/p^{e'}})$$ such that $\phi(x^{\alpha/p^e})=x^{{\alpha}'/p^{e'}}$.

Since $x^{\beta} \not \in \fq_{e'}$, there exists $\psi \in \Hom_R((S/J_{{\alpha}'}) x^{{\alpha}'/p^{e'}},R)$ such that $\psi(x^{{\theta}'} x^{{\alpha}'/p^{e'}})\not \in \fq$ by Proposition \ref{propo 2}. 

We have an $R$-linear map $$\varphi:R^{1/q}\longrightarrow \bigoplus_{\substack{1 \leq i \leq s\\ \alpha \in \mathcal{A}}}S/J_{\alpha}(a_ix^\alpha)^{1/q}$$ such that $$\varphi(r^{1/q})=\bigoplus_{\substack{1 \leq i \leq s\\ \alpha \in \mathcal{A}}}(r_{i,\alpha}+J_{\alpha})(a_ix^\alpha)^{1/q},$$ where $$r^{1/q} = \bigoplus_{\substack{1 \leq i \leq s\\ \alpha \in \mathcal{A}}}r_{i,\alpha}(a_ix^\alpha)^{1/q}.$$ Taking $\gamma=\psi \circ \phi \circ \pi_{\alpha} \circ \varphi$, we have $\gamma \in \Hom_R(R^{1/q},R)$, and $\gamma(x^{\beta/p^e})=\psi(x^{\theta} x^{{\alpha}'/p^{e'}})=\psi(x^{\tau}x^{{\theta}'} x^{{\alpha}'/p^{e'}})=x^{\tau}\psi(x^{{\theta}'} x^{{\alpha}'/p^{e'}})$.

In addition, $x^{\beta}=x^{q\theta} x^{\alpha}=x^{q\tau}x^{q{\theta}'} x^{\alpha}$. As $x^{\beta} \not \in \fq^{[q]}$, we get that $x^{\tau} \not \in \fq$. Since $x^{\beta} \in \fq_e$, it follows that $x^{\tau}\psi(x^{{\theta}'} x^{{\alpha}'/p^{e'}})=\gamma(x^{\beta/p^e}) \in \fq$. We get a contradiction, because $\fq$ is a prime ideal in $R$, and $x^{\tau},\;\psi(x^{{\theta}'} x^{{\alpha}'/p^{e'}}) \not \in \fq$.    
\end{proof}

\begin{proposition} \label{prodriveQ_e}
Let $e$ be a nonnegative integer, $q=p^e$, $\overline{R}=R/{\mathcal{P}(\fq)}$ with $\fq$ a monomial prime ideal in $R$, and $f \in R$. Then, the following hold.
\begin{itemize}
\item[(1)] If $f\in \fq_e$, then $\overline{f} \in (\overline{\fq})_e$;
\item[(2)] $\overline{f} \in \overline{\fq}^{[q]}$ if and only if $f\in \fq_e$.
\end{itemize} 
\end{proposition}
\begin{proof}
We show Part $(1)$. We can assume that $f$ a monomial, because $ \fq_e$ and $(\overline{\fq})_e$ are monomial ideals by Lemma \ref{monomial}. 

We have that $f \in \fq_e = \fq^{[q]}+\mathcal{P}(\fq)$ by Proposition \ref{proQ_e}. Since $f$ is a monomial, it follows that $f \in \fq^{[q]}$ or $f\in \mathcal{P}(\fq)$. If $f\in \mathcal{P}(\fq)$, then $\overline{f}=0 \in (\overline{\fq})_e$. Moreover, if $f \in \fq^{[q]}$, then $\overline{f} \in \overline{\fq}^{[q]} \subseteq (\overline{\fq})_e$. 

Now, we show Part $(2)$. From Proposition \ref{proQ_e}, we see that
\begin{align*}
\overline{f} \in \overline{\fq}^{[q]}=\overline{\fq^{[q]}}& \Leftrightarrow
 f-g \in \mathcal{P}(\fq) \; \mathrm{for \; some}\; g \in  \fq^{[q]}\\ 
 &\Leftrightarrow f \in \fq^{[q]}+\mathcal{P}(\fq)=\fq_e. 
\end{align*}
\end{proof}

\begin{proposition}\label{proQ_eseries}
Suppose $A$ as in Remark \ref{remarkdimension} and $B=A/IA$. Given $\fq$ a monomial prime ideal of $B$, then $\fq_e=\fq^{[q]}+\mathcal{P}(\fq)$ for every $e \in \mathbb{N}$, and $q=p^e$.
\end{proposition}
\begin{proof}
The proof is analogous to Proposition \ref{proQ_e}.
\end{proof}

\begin{proposition} \label{proderiveQ_eseries}
Suppose $A$ as in Remark \ref{remarkdimension} and $B=A/IA$. Let $e$ be a nonnegative integer, $q=p^e$, $\overline{B}=B/{\mathcal{P}(\fq)}$ with $\fq$ a monomial prime ideal in $B$, and $f \in B$. Then, the following hold.
\begin{itemize}
\item[(1)] If $f\in \fq_e$, then $\overline{f} \in (\overline{\fq})_e$;
\item[(2)] $\overline{f} \in \overline{\fq}^{[q]}$ if and only if $f\in \fq_e$.
\end{itemize} 
\end{proposition}
\begin{proof}
The proof is analogous to Proposition \ref{prodriveQ_e}.
\end{proof}

\section{Cartier Threshold of $\mathfrak{a}$ with Respect to $J$}

In this section we prove another one of our main results, Theorem \ref{MainThmA}. In order to obtain this, we recall the definition of the Cartier threshold of $\mathfrak{a}$ with respect to $J$. We give some properties of this and show that it is preserved under localization and completion. We study its relation with the $F$-thresholds. We also compare this number with its analog in $\overline{R}=R/\mathcal{P}(J)$.
 
\subsection{Definition and First Properties}

In this subsection $R$ denotes a ring of prime characteristic $p$. We give properties related to Cartier thresholds. 

\begin{definition}[{\cite{DSHNBW}}]
Let $R$ be an $F$-finite $F$-pure ring. Given $\mathfrak{a}$, $J$ two ideals in $R$ such that $\mathfrak{a} \subseteq \sqrt{J}$, we define 
\begin{align*}
b_{\mathfrak{a}}^J(p^e)&= \max \{t \in \mathbb{N}\;|\; \mathfrak{a}^t \not \subseteq J_e\}.
\end{align*}
 
We define the Cartier threshold of $\mathfrak{a}$ in $R$ with respect to $J$ by
\begin{align*}
\ct_J(\mathfrak{a})&= \lim\limits_{e \rightarrow \infty} {\frac{b_{\mathfrak{a}}^J(p^e)}{p^e}}.
\end{align*}

If $(R,\mathfrak{m},K)$ is a local ring or a standard graded $K$-algebra and $\mathfrak{m}=J$, the Cartier threshold $\ct_J(\mathfrak{a})$ coincides with the $F$-pure threshold $\fpt(\mathfrak{a})$. When $\mathfrak{a}=\mathfrak{m}$, $\fpt(\mathfrak{m})$ is denoted by $\fpt(R)$.   
\end{definition}

Using Proposition \ref{pro6}, it follows that $ \ct_J(\mathfrak{a})$ also commutes with arbitrary intersections. 

\begin{proposition} \label{teo6} 
Let $R$ be an $F$-finite $F$-pure ring. Let $\{\fq_i\}_i$ be a family of ideals in $R$, and $J=\bigcap_i\fq_i$. Let $\mathfrak{a}$ be an ideal in $R$ such that $\mathfrak{a} \subseteq \sqrt{J}$. Then, $\ct_J(\mathfrak{a})=\sup\{\ct_{\fq_i}(\mathfrak{a})\}$. 
\end{proposition}
\begin{proof}
By Proposition \ref{pro6}, we have that $J_e=\bigcap_i(\fq_i)_e$ for every nonnegative integer $e$. Then,
\begin{align*} 
t \geq b_{\mathfrak{a}}^{J}(p^e)& \Leftrightarrow \mathfrak{a}^{t+1} \subseteq J_e\\
 & \Leftrightarrow \mathfrak{a}^{t+1} \subseteq (\fq_i)_e \; \mathrm{for \; every} \;i\\
  & \Leftrightarrow t \geq b_{\mathfrak{a}}^{\fq_i}(p^e) \; \mathrm{for \;every} \; i\\
   & \Leftrightarrow t \geq \sup\{b_{\mathfrak{a}}^{\fq_i}(p^e)\}.
\end{align*}

Hence, $\frac{ b_{\mathfrak{a}}^{J}(p^e)}{p^e}=\sup\left\{\frac{b_{\mathfrak{a}}^{\fq_i}(p^e)}{p^e}\right\}$. Therefore, $\ct_J(\mathfrak{a})=\sup\{\ct_{\fq_i}(\mathfrak{a})\}$.      
\end{proof}

Since $\fq_e$ is a $\fq$-primary ideal by Proposition \ref{primary}, we have that $\ct_J(\mathfrak{a})$ is preserved under localization. This fact, we prove it in Proposition \ref{teo9} below.

\begin{lemma}\label{lem5}
Let $R$ be an $F$-finite $F$-pure ring, $\fq$ be a prime ideal of $R$, and $f \in R$. Then, $\frac{f}{1} \in I_e(R_\fq)$ if and only if $f \in \fq_e$.
\end{lemma}
\begin{proof}
We focus on the first direction. Let $\psi \in \Hom_R(R^{1/p^e},R)$. Since $(R^{1/p^e})_\fq \cong {R_\fq}^{1/p^e}$ as $R_\fq$-module, $\psi_\fq \in \Hom_{R_\fq}({R_\fq}^{1/p^e},R_\fq)$, and so, $\frac{\psi(f^{1/p^e})}{1}=\psi_\fq(\frac{f^{1/p^e}}{1})=\psi_\fq((\frac{f}{1})^{1/p^e}) \in \fq R_\fq$. Hence, as $\fq$ is a prime ideal, $\psi(f^{1/p^e})\in \fq$. Therefore, $f\in \fq_e$.

We now show the other direction. Let $\psi \in \Hom_{R_\fq}({R_\fq}^{1/p^e},R_\fq)$. Since $\Hom_{R_\fq}({R_\fq}^{1/p^e},R_\fq) \cong \Hom_R(R^{1/p^e},R)_{\fq}$, we have that $\psi = \frac{1}{s}  \varphi_{\fq}$ for some $\varphi \in\Hom_R(R^{1/p^e},R)$ and $s \in R \setminus \fq$. As a consequence, $\psi((\frac{f}{1})^{1/p^e})=\psi(\frac{f^{1/p^e}}{1})=\frac{\varphi(f^{1/p^e})}{s} \in  \fq R_\fq$. Therefore, $\frac{f}{1} \in I_e(R_\fq)$. 

\end{proof}
\begin{proposition}\label{teo9}
Let $R$ be an $F$-finite $F$-pure ring. Let $\mathfrak{a},\; \fq$ be two ideals of R with $\fq$ a prime ideal, and $\mathfrak{a} \subseteq \fq$. Then, $\ct_{\fq}(\mathfrak{a})=\fpt(\mathfrak{a}R_\fq)$.   
\end{proposition}
\begin{proof}
By Lemma \ref{lem5}, we observe that,
\begin{align*}
b_{\mathfrak{a}}^{\fq}(p^e)&=\max \{t \in \mathbb{N} \; |\; \mathfrak{a}^t \not \subseteq \fq_e \} \\ 
&=\max \{t \in \mathbb{N}\; |\; {\mathfrak{a}^t}R_\fq  \not \subseteq I_e(R_\fq) \}\\ 
&=\max \{t \in \mathbb{N}\; |\; (\mathfrak{a}R_\fq)^t  \not \subseteq I_e(R_\fq) \} \\
&=b_{\mathfrak{a}R_\fq}^{\fq R_\fq}(p^e).
\end{align*}
Therefore, $\ct_{\fq}(\mathfrak{a})=\fpt(\mathfrak{a}R_\fq)$.  
\end{proof} 

Consider a local ring $(R,\mathfrak{m},K)$. Let $\mathfrak{a} \subseteq  \sqrt{J}$ be two ideals of $R$. We claim that the Cartier threshold of $\mathfrak{a} $ with respect to $J$ does not vary under completion. To show this, we compare the ideal $J_e$ versus $(J \widehat{R})_e $.

\begin{lemma}\label{lem6}
Let $(R,\mathfrak{m},K)$ be an $F$-finite $F$-pure local ring, $f\in R$, and $J$ be an ideal in $R$. Then, $f\in J_e$ if and only if $f\in (J\widehat{R})_e$.
\end{lemma}
\begin{proof}
We suppose that $f\in J_e$. Let $\varphi \in \Hom_{\widehat{R}}(\widehat{R}^{1/p^e},\widehat{R})$. Since $R$ is an $F$-finite ring and  $\widehat{R}^{1/p^e} \cong \widehat{R^{1/p^e}}$ as $\widehat{R}$-module, we have
\begin{align*}
\Hom_{\widehat{R}}(\widehat{R}^{1/p^e},\widehat{R})&\cong 
 \lhat{78}{\Hom_R(R^{1/p^e},R)}\\
&\cong \Hom_R(R^{1/p^e},R)\otimes_R \widehat{R}. 
\end{align*}
Hence, $\varphi=\sum_{i=1}^{n}\varphi_i \otimes r_i$ with $\varphi_i \in \Hom_R(R^{1/p^e},R)$ and $r_i \in \widehat{R}$. Then, $\varphi(f^{1/p^e})=\sum_{i=1}^{n}r_i\varphi_i(f^{1/p^e})$. However, $f\in J_e$, in consequence $\varphi_i(f^{1/p^e}) \in J$, thus $\varphi(f^{1/p^e})\in J\widehat{R}$. Therefore, $f\in (J\widehat{R})_e$.

We now suppose that $f\in (J\widehat{R})_e$. Let $\varphi \in \Hom_R(R^{1/p^e},R)$. Since $\widehat{R}^{1/p^e} \cong \widehat{R^{1/p^e}}$ as $\widehat{R}$-module, we have $\widehat{\varphi} \in \Hom_{\widehat{R}}(\widehat{R}^{1/p^e},\widehat{R})$. Then, $\widehat{\varphi}(f^{1/p^e})\in J\widehat{R}$, and so, $\varphi(f^{1/p^e})\in J$. Therefore, $f\in J_e$.           
\end{proof}
\begin{proposition} \label{teo10}
Suppose that $(R,\mathfrak{m},K)$ is an $F$-finite $F$-pure local ring. Let $\mathfrak{a}$, $J$ be two ideals in $R$ such that $\mathfrak{a} \subseteq \sqrt{J}$. Then, $\ct_J(\mathfrak{a})=\ct_{J\widehat{R}}(\mathfrak{a} \widehat{R})$.  
\end{proposition}
\begin{proof}
By Lemma \ref{lem6}, we observe that
\begin{align*}
b_{\mathfrak{a}}^J(p^e)&=\max \{t \in \mathbb{N} \; |\; \mathfrak{a}^t \not \subseteq J_e \} \\ 
&=\max \{t \in \mathbb{N}\; |\; {\mathfrak{a}^t}\widehat{R} \not \subseteq (J\widehat{R})_e \}\\ 
&=\max \{t \in \mathbb{N}\; |\; (\mathfrak{a}\widehat{R})^t  \not \subseteq (J\widehat{R})_e \} \\
&=b_{\mathfrak{a}\widehat{R}}^{J\widehat{R}}(p^e).
\end{align*}
Therefore, $\ct_J(\mathfrak{a})=\ct_{J\widehat{R}}(\mathfrak{a} \widehat{R})$. 
\end{proof}

Given $J$ an ideal in $R$, we consider the ring $\overline{R}=R/\mathcal{P}(J)$. Let $\mathfrak{a}$ be an ideal in $R$ such that $\mathfrak{a} \subseteq \sqrt{J}$. Our goal is to compare the Cartier threshold of $\mathfrak{a}$ with respect to $J$ versus the Cartier threshold of $\overline{\mathfrak{a}}$ with respect to $\overline{J}$.

\begin{lemma} \label{lema4}
Let $R$ be an $F$-finite $F$-pure ring, $J$ be an ideal of $R$, $\overline{R}=R/\mathcal{P}(J)$, and $f \in R$. Then,  $\overline{f} \in (\overline{J})_e$ implies that $f \in J_e$.  
\end{lemma}
\begin{proof}
For every $\varphi \in \Hom_R(R^{1/p^e},R)$, we take $\overline{\varphi}:\overline{R}^{1/p^e} \longrightarrow \overline{R}$ such that $\overline{\varphi}(\overline{x}^{1/p^e})=\overline{\varphi(x^{1/p^e})}$. By Lemma \ref{lema6}, it follows that $\overline{\varphi}$ is well defined. 

Since $\varphi \in \Hom_R(R^{1/p^e},R)$, it follows that $\overline{\varphi} \in \Hom_{\overline{R}}(\overline{R}^{1/p^e},\overline{R})$. As $\overline{f} \in (\overline{J})_e$, then $\overline{\varphi(f^{1/p^e})}=\overline{\varphi}(\overline{f}^{1/p^e}) \in \overline{J}$. Hence, there exists $y \in J$ such that $\varphi(f^{1/p^e})-y \in \mathcal{P}(J) \subseteq J$, and so  $\varphi(f^{1/p^e}) \in J$. Therefore, $f \in J_e$.       
\end{proof}
\begin{proposition} \label{tore4}
Let $R$ be an $F$-finite $F$-pure ring. Let $\mathfrak{a}$, $J$ be two ideals in $R$ such that $\mathfrak{a} \subseteq \sqrt{J}$, and $\overline{R}=R/\mathcal{P}(J)$. Then, $\ct_J(\mathfrak{a})\leq \ct_{\overline{J}}(\overline{\mathfrak{a}})$. In particular, if $(R,\mathfrak{m},K)$ is a local ring or a standard graded $K$-algebra, then $\fpt(\mathfrak{a})\leq \fpt(\overline{\mathfrak{a}})$. 
\end{proposition}
\begin{proof}
From Lemma \ref{lema4}, we have that
\begin{align*}
b_{\mathfrak{a}}^J(p^e)&=\max\{t \in \NN\;|\;\mathfrak{a}^t \not \subseteq J_e\} \\
&\leq \max\{t \in \NN\;|\;\overline{\mathfrak{a}}^t \not \subseteq (\overline{J})_e\} \\
&=b_{\overline{\mathfrak{a}}}^{\overline{J}}(p^e).
\end{align*}
Therefore, $\ct_J(\mathfrak{a})=\lim\limits_{e \rightarrow \infty} {\frac{b_{\mathfrak{a}}^J(p^e)}{p^{e}}} \leq \lim\limits_{e \rightarrow \infty} {\frac{b_{\overline{\mathfrak{a}}}^{\overline{J}}(p^e)}{p^{e}}}=\ct_{\overline{J}}(\overline{\mathfrak{a}})$.
\end{proof}


\subsection{Relation Between $c^J(\mathfrak{a})$ and $\ct _J(\mathfrak{a})$}

In this subsection we give a characterization of $\ct_J(\mathfrak{a})$ using $F$-thresholds.
 
\begin{remark} \label{pro4} 
Suppose that $R$ is an $F$-finite $F$-pure ring. Let $\mathfrak{a}$, $J$ be two ideals in $R$ such that $\mathfrak{a} \subseteq \sqrt{J}$. Since $J^{[p^e]} \subseteq J_e$, we have that
\begin{align*}
b_{\mathfrak{a}}^J(p^e)&=\max\{t \in \NN\;|\;\mathfrak{a}^t \not \subseteq J_e\} \\
&\leq \max\{t \in \NN\;|\;\mathfrak{a}^t \not \subseteq J^{[p^e]}\} \\
&=\nu^{J}_{\mathfrak{a}}(p^{e}).
\end{align*}
Therefore, $\ct_J(\mathfrak{a})\leq c^{J}(\mathfrak{a})$. 
\end{remark}

The following Remark relates $F$-pure rings and Frobenius powers.

\begin{remark}\label{increasing}
Suppose that $R$ is an $F$-pure ring. Let $J$ be an ideal in $R$ and $r \in R$. Then, $r^{p}\in J^{[p]}$ if and only if $r \in J$ \cite{Fedder'scriterion}. Let $\mathfrak{a}$ be an ideals in $R$ such that $\mathfrak{a} \subseteq \sqrt{J}$. Then, we have that $\fa^{\nu_{\fa}^{J}(p^{e})}  \nsubseteq J^{[p^{e}]}$.  As a consequence, $(\fa^{\nu_{\fa}^{J}(p^{e})})^{[p]}  \nsubseteq J^{[p^{e+1}]}$. Hence, $\fa^{p \cdot \nu_{\fa}^{J}(p^{e})}  \nsubseteq J^{[p^{e+1}]}$, and so, $p \cdot \nu_{\fa}^{J}(p^{e})\leq \nu_{\fa}^{J}(p^{e+1})$. Therefore, the sequence $\left\lbrace \frac{\nu_{\fa}^{J}(p^{e})}{p^{e}} \right\rbrace_{e \geq 0}$ is non-decreasing. 
\end{remark}  

The following propositions are an extension of the work done by De Stefani, N\'u\~nez-Betancourt and P\'erez \cite[Theorem $4.6$]{DSNBP}.

\begin{proposition}
Let $R$ be an $F$-finite $F$-pure ring. Let $J$ be an ideal in $R$. Then, $J_e^{[p]} \subseteq J_{e+1}$ for every $e \in \mathbb{N}$.
\end{proposition}
\begin{proof}
Let $f$ be an element in $J_e$. Let $\varphi \in \Hom_R(R^{1/p^{e+1}},R)$. As $R^{1/p^e} \subseteq R^{1/p^{e+1}}$, we have that $\varphi|_{R^{1/p^e}} \in \Hom_R(R^{1/p^e},R)$. Thus, $\varphi((f^p)^{1/p^{e+1}})=\varphi|_{R^{1/p^e}}(f^{1/p^e}) \in J$. Hence, $f^p \in J_{e+1}$, and so, $J_e^{[p]} \subseteq J_{e+1}$.  
\end{proof}
\begin{proposition}
Let $R$ be an $F$-finite $F$-pure ring, and $\mathfrak{a}$, $J$ be two ideals in $R$ such that $\mathfrak{a} \subseteq \sqrt{J}$. The sequence $ \left\lbrace \frac{c^{J_e}(\mathfrak{a})}{p^e} \right\rbrace_{e \geq 0}$ is non-increasing and bounded below by zero. In particular, its limit exists. 
\end{proposition}
\begin{proof}
Let $e$ be nonnegative integer, $J_e^{[p]} \subseteq J_{e+1}$. Thus, $c^{J_{e+1}}(\mathfrak{a}) \leq c^{J_e^{[p]}}(\mathfrak{a})=p \cdot c^{J_e}(\mathfrak{a})$ by Proposition \ref{cpFthreshold}. Therefore, $\frac{c^{J_{e+1}}(\mathfrak{a})}{p^{e+1}} \leq \frac{c^{J_e}(\mathfrak{a})}{p^e}$.   
\end{proof}

The following proposition gives us a relation between the Cartier thresholds and $F$-thresholds. Specifically, we can obtain the Cartier threshold as a limit $F$-thresholds.
\begin{proposition} \label{teo7}
Let $R$ be an $F$-finite $F$-pure ring. Let $\mathfrak{a}$, $J$ be two ideals in $R$ such that $\mathfrak{a} \subseteq \sqrt{J}$. Then, $\ct_J(\mathfrak{a})=\lim\limits_{e \rightarrow \infty} {\frac{c^{J_e}(\mathfrak{a})}{p^e}}$. 
\end{proposition}
\begin{proof}
Let $e$ be nonnegative integer. We note that 
\begin{align*}
b_{\mathfrak{a}}^J(p^e)&=\max\{t\in \mathbb{N}\;|\; \mathfrak{a} \not \subseteq J_e\}\\
&=\max\{t\in \mathbb{N}\;|\; \mathfrak{a} \not \subseteq J_e^{[p^0]}\}\\
&=\nu^{J_e}_{\mathfrak{a}}(p^0).
\end{align*}
For every nonnegative integer $s$, we have
\begin{align*}
\frac{\nu^{J_e}_{\mathfrak{a}}(p^s)}{p^s}-\frac{\nu^{J_e}_{\mathfrak{a}}(p^0)}{p^0} \leq \frac{\mu(\mathfrak{a})}{p^0}
\end{align*}
by Lemma \ref{lem4}.

The sequence $ \left\lbrace \frac{\nu_{\mathfrak{a}}^{J_e}(p^s)}{p^s} \right\rbrace_{s \geq 0}$ is non-decreasing by Remark \ref{increasing}. As a consequence,    
\begin{align*}
0 \leq \frac{\nu^{J_e}_{\mathfrak{a}}(p^s)}{p^s}-\nu^{J_e}_{\mathfrak{a}}(p^0) \leq \mu(\mathfrak{a}).
\end{align*}
Thus,
\begin{align*}
0 \leq \frac{\nu^{J_e}_{\mathfrak{a}}(p^s)}{p^s}-b_{\mathfrak{a}}^J(p^e) \leq \mu(\mathfrak{a}).
\end{align*}
We take limit over $s$ to get
\begin{align*}
0 \leq c^{J_e}(\mathfrak{a})-b_{\mathfrak{a}}^J(p^e) \leq \mu(\mathfrak{a}),
\end{align*}
dividing by $p^e$ gives
\begin{align*}
0 \leq \frac{c^{J_e}(\mathfrak{a})}{p^e}-\frac{b_{\mathfrak{a}}^J(p^e)}{p^e} \leq \frac{\mu(\mathfrak{a})}{p^e}.
\end{align*}
Taking limit over $e$ we conclude that
\begin{align*}
 \ct_J(\mathfrak{a})=\lim\limits_{e \rightarrow \infty} {\frac{c^{J_e}(\mathfrak{a})}{p^e}}.
\end{align*}
\end{proof}
\begin{corollary}
Let $R$ be an $F$-finite $F$-pure ring. Let $\mathfrak{a}$, $J$ be two ideals in $R$ such that $\mathfrak{a} \subseteq \sqrt{J}$. Then, $\ct_J(\mathfrak{a})=c^{J}(\mathfrak{a})$ if and only if $c^{J_e}(\mathfrak{a})=c^{J^{[p^e]}}(\mathfrak{a})$ for every $e \in \mathbb{N}$. 
\end{corollary}
\begin{proof}
We focus on the first direction, it suffices to show $c^{J^{[p^e]}}(\mathfrak{a}) \leq c^{J_e}(\mathfrak{a})$. As the sequence $ \left\lbrace \frac{c^{J_e}(\mathfrak{a})}{p^e} \right\rbrace_{e\geq 0}$ is non-increasing and bounded below, it converges to its infimum. By Proposition \ref{teo7}, $c^{J}(\mathfrak{a}) \leq \frac{c^{J_e}(\mathfrak{a})}{p^e}$. As a consequence, $c^{J^{[p^e]}}(\mathfrak{a})= p^e \cdot c^{J}(\mathfrak{a}) \leq c^{J_e}(\mathfrak{a})$.

We now show the other direction, $\ct_J(\mathfrak{a})=\lim\limits_{e \rightarrow \infty} {\frac{c^{J_e}(\mathfrak{a})}{p^e}}=\lim\limits_{e \rightarrow \infty} {\frac{c^{J^{[p^e]}}(\mathfrak{a})}{p^e}}=\lim\limits_{e \rightarrow \infty} {\frac{p^e \cdot c^{J}(\mathfrak{a})}{p^e}}= c^{J}(\mathfrak{a})$.   
\end{proof}

\subsection{Cartier Thresholds in Stanley-Reisner Rings}
Throughout this subsection, we denote $S=K[x_1,  \ldots ,x_n]$ with $K$ an $F$-finite field of prime characteristic $p$. Let $I$ be a squarefree monomial ideal of $S$, $R=S/I$, and $\fp_1, \ldots,\fp_l$ are the minimal prime ideals of $R$. 


\begin{theorem} \label{pro1series} 
Suppose $A$ as in Remark \ref{remarkdimension} and $B=A/IA$. Let $\mathfrak{a}$, $\fq$ be two ideals in $B$ with $\fq$ a prime monomial ideal, such that $\mathfrak{a} \subseteq \fq$, and $\overline{B}=B/{\mathcal{P}(\fq)}$. Then, the following hold: 
\begin{itemize}
\item[(1)]$\ct_\fq(\mathfrak{a})=\ct_{\overline{\fq}}(\overline{\mathfrak{a}})$;
\item[(2)] $\ct_\fq(\mathfrak{a})=c^{\overline{\fq}}(\overline{\mathfrak{a}})$;
\item[(3)] $\ct_\fq(\mathfrak{a})$ is a rational number.  
\end{itemize}

In particular, $\fpt(\mathfrak{a})$ is a rational number.
\end{theorem}
\begin{proof}
We show Part $(1)$. From Proposition \ref{proderiveQ_eseries} and Lemma \ref{lema4}, we have
\begin{align*}
b_{\mathfrak{a}}^\fq(p^e)&=\max\{t\in \mathbb{N} \;|\;\mathfrak{a}^t \not\subseteq \fq_e\} \\
&=\max\{t\in \mathbb{N} \;|\;\overline{\mathfrak{a}}^t \not\subseteq (\overline{\fq})_e\} \\
&= b_{\overline{\mathfrak{a}}}^{\overline{\fq}}(p^e).
\end{align*}
Therefore, $\ct_\fq(\mathfrak{a})=\ct_{\overline{\fq}}(\overline{\mathfrak{a}})$.

Now, we show Part $(2)$. We claim that $c^{\overline{\fq}}(\overline{\mathfrak{a}}) \leq \ct_\fq(\mathfrak{a})$. From Proposition \ref{proderiveQ_eseries}, it follows that
\begin{align*}
\nu_{\overline{\mathfrak{a}}}^{\overline{\fq}}(p^e)&=\max\{t\in \mathbb{N} \;|\;\overline{\mathfrak{a}}^t \not\subseteq \overline{\fq}^{[q]}\} \\
&\leq \max\{t\in \mathbb{N} \;|\;\mathfrak{a}^t \not\subseteq \fq_e\} \\
&=b_{\mathfrak{a}}^\fq(p^e).
\end{align*}

Thus, $c^{\overline{\fq}}(\overline{\mathfrak{a}})=\lim\limits_{e \rightarrow \infty} {\frac{\nu_{\overline{\mathfrak{a}}}^{\overline{\fq}}(p^e)}{p^{e}}} \leq \lim\limits_{e \rightarrow \infty} {\frac{b_{\mathfrak{a}}^{\fq}(p^e)}{p^{e}}}=\ct_\fq(\mathfrak{a})$.

By Part $(1)$ and Remark \ref{pro4}, we have $c^{\overline{\fq}}(\overline{\mathfrak{a}}) \leq \ct_\fq(\mathfrak{a})=\ct_{\overline{\fq}}(\overline{\mathfrak{a}}) \leq c^{\overline{\fq}}(\overline{\mathfrak{a}})$. Therefore, $c^{\overline{\fq}}(\overline{\mathfrak{a}})=\ct_\fq(\mathfrak{a})$.
 
We show Part $(3)$. Since $\fq$ is a monomial ideal, ${\mathcal{P}(\fq)}$ is also a monomial ideal by Lemma \ref{monomial}. In addition, ${\mathcal{P}(\fq)}$ is a radical ideal by Remark \ref{remarkCartierRadical}. Thus, ${\mathcal{P}(\fq)}$ is squarefree monomial ideal. Consequently, $\overline{B}$ is a power series ring modulo a squarefree monomial ideal. Since  $\overline{\fq}$ is a monomial ideal in $\overline{B}$, $c^{\overline{\fq}}(\overline{\fa})$ is a rational number by Remark \ref{remarkseries}. Therefore, $\ct_\fq(\mathfrak{a})$ is a rational number by Part $(2)$. 

The last statement follows, since $\ct_\m(\mathfrak{a})=\fpt(\fa)$ and $\m$ is a monomial prime ideal in $B$.
\end{proof}

Since $ \ct_J(\mathfrak{a})$ is preserved under localization and completion, Theorem \ref{pro1series} allows us to obtain one of the main results of this work.  


\begin{corollary}\label{mainresult}
Let $\mathfrak{a}$, $J$ be two ideals in $R$ with $J$ radical ideal, such that $\mathfrak{a} \subseteq J$. Then, $\ct_J(\mathfrak{a})$ is a rational number. In particular, $\fpt(\fa)$ is a rational number. 
\end{corollary}
\begin{proof}
Let $\fq \subseteq R$ be a prime ideal such that $\fa \subseteq \fq$. We have that $\ct_\fq(\mathfrak{a})=\fpt(\mathfrak{a}\widehat{R_\fq})$ by Propositions \ref{teo9} and \ref{teo10}. Thus, $\ct_\fq(\mathfrak{a})$ is a rational number by Theorem \ref{pro1series}.

Since $J$ is a radical ideal, we have that $J=\bigcap_{i=1}^m \fq_i$ where $\fq_1, \ldots, \fq_m$ are the minimal prime ideals of $J$. From Proposition \ref{teo6}, $\ct_J(\mathfrak{a})=\max\{\ct_{\fq_i}(\mathfrak{a})\}$. Therefore, $\ct_J(\mathfrak{a})$ is a rational number.  
\end{proof}

\section{$a$-Invariants and Regularity}
In this section we focus on standard graded $K$-algebras. We study the $a$-invariants and regularity in rings modulo Frobenius powers of an ideal. We also investigate their behavior with the Castelnuovo-Mumford regularity in Stanley-Reisner rings.

\subsection{Definitions and Property}

Suppose that $(R,\m, K)$ is a standard graded $K$-algebra, and let $I$ be a homogeneous ideal of $R$. We recall that if $M$ is a graded $R$-module, its $i$-th local cohomology $H^i_I(M)$ is a graded module. Moreover, if $M$ is finitely generated, the module $H^i_{\m}(M)$ is Artinian. Therefore, one can define the following number.
\begin{definition}[{\cite{GW1}}]
Let $(R,\m, K)$ be a standard graded $K$-algebra. Let $M$ be an $\frac{1}{p^e}\NN$-graded finitely generated $R$-module. If $H^i_{\m}(M) \not = 0$, we define the $i$-th $a$-invariant of $M$ by 
\begin{align*}
a_i(M)=\max \left\{s\in \frac{1}{p^e}\ZZ\;|\; H^i_{\m}(M)_s \not =0 \right\}.
\end{align*}
If $H^i_{\m}(M) = 0$, we set $a_i(M)=-\infty$.    
\end{definition}
\begin{definition}
Let $(R,\m, K)$ be a standard graded $K$-algebra. Let $M$ be an $\frac{1}{p^e}\NN$-graded finitely generated $R$-module. We define the regularity of $M$ by 
\begin{align*}
\reg(M)=\max_{i\in \ZZ}\{a_i(M)+i\}.
\end{align*}
\end{definition}

Next theorem gives us conditions for the regularity in rings modulo Frobenius power of ideals. 
\begin{theorem}[{\cite[Theorem $5.4$]{DSNBP}}]
Let $(R,\m, K)$ be a standard graded $K$-algebra that is $F$-finite and $F$-pure. Suppose that $J$ is a homogeneous ideal of $R$. If there exists a constant $C$ such that $\reg(R/J^{[p^e]})\leq Cp^e$ for all $e \gg 0$, then 
\begin{align*}
\lim\limits_{e \rightarrow \infty} {\frac{\reg(R/J^{[p^e]})}{p^{e}}}
\end{align*}
exists, and it is bounded below by $\max_{i \in \NN}\{a_i(R/J)\}+\fpt(R)$.
\end{theorem}

\subsection{Regularity in Stanley-Reisner Rings}
Throughout this subsection, we denote $S=K[x_1,  \ldots , x_n]$ with $K$ an $F$-finite field of prime characteristic $p$. Let $I$ be a squarefree monomial ideal of $S$, $R=S/I$.
\begin{definition}
Let $N \subseteq \{1,\ldots,n\}$. We define $$x^{N}=\prod_{i \in N}x_{i}.$$
\end{definition}
Now, we define the support of an element in $\NN^{n}$.
\begin{definition}
Let $\alpha \in \NN^n$. The support of $\alpha$ is defined by 
\begin{align*}
\Supp(\alpha)=\{i \in \{1,\ldots,n\}\, | \; \alpha_i \not = 0\}.
\end{align*}
\end{definition}

\begin{lemma}\label{lemsuport}
Given $\alpha \in \NN^n$, then $(I:x^{\alpha})=(x^{\Supp(\lambda)\setminus\Supp(\alpha)}\;|\;x^\lambda \;\mathrm{minimal \; generator \; of}\; I)$. In particular, if $\alpha, \beta \in \NN^n$ are such that $\Supp(\alpha)=\Supp(\beta)$, then $(I:x^{\alpha})=(I:x^{\beta})$.
\end{lemma}
\begin{proof}
Since $I$ is a monomial ideal, it follows that $(I:x^{\alpha})$ is a monomial ideal as well.
We have $(x^{\Supp(\lambda) \setminus \Supp(\alpha)}\;|\;x^\lambda \;\mathrm{minimal \; generator \; of}\; I) \subseteq (I:x^{\alpha})$. Indeed, for every $x^\lambda$ minimal generator of $I$, $x^{\Supp(\lambda) \setminus \Supp(\alpha)}x^\alpha \in I$.  

We show that $(I:x^{\alpha}) \subseteq (x^{\Supp(\lambda) \setminus \Supp(\alpha)}\;|\;x^\lambda \;\mathrm{minimal \; generator \; of}\; I)$. Let $x^{\theta}$ be a generator of $(I:x^{\alpha})$. Thus $x^{\theta}x^{\alpha} \in I$. Hence, $x^{\lambda}|x^{\theta}x^{\alpha}$ for some $x^{\lambda}$ minimal generator of $I$. Then, $\Supp(\lambda)\setminus\Supp(\alpha)\subseteq \Supp(\theta)$, and so, $x^{\Supp(\lambda) \setminus \Supp(\alpha)}|x^\theta$. Therefore, $x^\theta \in (x^{\Supp(\lambda) \setminus \Supp(\alpha)}\;|\;x^\lambda \;\mathrm{minimal \; generator \; of}\; I)$.
\end{proof}    

Now, we prove Theorem \ref{MainThmC}.
\begin{theorem}\label{thmregularity}
Let $J$ be a homogeneous ideal of $R$. Then, 
\begin{align*}
\lim\limits_{e \rightarrow \infty} {\frac{\reg(R/J^{[p^e]})}{p^{e}}}&=\max_{\substack{1 \leq i \leq d\\ \alpha \in \mathcal{A}'}}\{a_i(S/(J_\alpha+J))+|\alpha|\},
\end{align*}
where $\mathcal{A}'=\{\alpha \in \mathbb{N}^n\;|\;0\leq \alpha_i\leq 1\;$for$\;i=1,\ldots,n\}$, $J_{\alpha}=(I:x^\alpha)$, and $d=\max\{\dim(S/(J_\alpha+J))\;|\;\alpha \in \cA'\}$. In particular, this limit is an integer number. 
\end{theorem}
\begin{proof}
Without loss of generality, we can take $K$ a perfect field. Let $e$ be a nonnegative integer and $\mathcal{A}=\{\alpha \in \mathbb{N}^n\;|\;0\leq \alpha_i\leq p^e-1\;$for$\;i=1,\ldots,n\}$. Then, 
\begin{align*}
R^{1/p^e} \cong \bigoplus_{ \alpha \in \mathcal{A}}(S/J_\alpha)x^{\alpha/p^e},
\end{align*}
where $J_{\alpha}=(I:x^\alpha)$ by Proposition \ref{propo 2}. Applying $-\otimes_R R/J$, we obtain that
\begin{align*}
(R/J^{[p^e]})^{1/p^e}  \cong R^{1/p^e}/J R^{1/p^e} \cong \bigoplus_{ \alpha \in \mathcal{A}}(S/(J_\alpha+J))x^{\alpha/p^e},
\end{align*}
and so
\begin{align*}
H^i_\m((R/J^{[p^e]})^{1/p^e})   \cong \bigoplus_{ \alpha \in \mathcal{A}}H^i_\m((S/(J_\alpha+J))x^{\alpha/p^e}).
\end{align*}
Hence, we have
\begin{align*}
\frac{a_i(R/J^{[p^e]})}{p^e}&=a_i((R/J^{[p^e]})^{1/p^e})\\
&=\max_{\alpha \in \cA}\{a_i((S/(J_\alpha+J))x^{\alpha/p^e})\}\\
&=\max_{\alpha \in \cA} \left\{a_i(S/(J_\alpha+J))+\frac{|\alpha|}{p^e} \right\}.
\end{align*}
From Lemma \ref{lemsuport}, we have 
\begin{align*}
\frac{a_i(R/J^{[p^e]})}{p^e}&=\max_{\alpha \in \cA'}\left\{a_i(S/(J_\alpha+J))+\frac{|\alpha|(p^e-1)}{p^e}\right\}.
\end{align*}
Thus, 
\begin{align*}
\lim\limits_{e \rightarrow \infty} {\frac{\reg(R/J^{[p^e]})}{p^{e}}}&=\lim\limits_{e \rightarrow \infty} {\max_{i\in \ZZ}\left\{\frac{a_i(R/J^{[p^e]})}{p^e}+\frac{i}{p^e}\right\}}\\
&=\lim\limits_{e \rightarrow \infty} {\max_{i\in \ZZ} \left\{ \max_{\alpha \in \cA'}\{a_i(S/(J_\alpha+J))+\frac{|\alpha|(p^e-1)}{p^e}\}+\frac{i}{p^e}\right\}}\\
&=\lim\limits_{e \rightarrow \infty} {\max_{\substack{1 \leq i \leq d\\ \alpha \in \mathcal{A}'}} \left\{a_i(S/(J_\alpha+J))+\frac{|\alpha|(p^e-1)}{p^e}+\frac{i}{p^e}\right\}}\\
&= \max_{\substack{1 \leq i \leq d\\ \alpha \in \mathcal{A}'}} \left\{\lim\limits_{e \rightarrow \infty}{a_i(S/(J_\alpha+J))+\frac{|\alpha|(p^e-1)}{p^e}+\frac{i}{p^e}}\right\}\\
&= \max_{\substack{1 \leq i \leq d\\ \alpha \in \mathcal{A}'}}\{a_i(S/(J_\alpha+J))+|\alpha|\}.
\end{align*}
 
\end{proof}

\section*{Acknowledgments}
I would like to thank Alessandro De Stefani, Daniel J. Hern\'andez, Luis N\'u\~nez-Betancourt, and Emily E. Witt for sharing their work \cite{DSHNBW}. I also thank them, and the referee for helpful comments, and suggestions.

\bibliographystyle{alpha}
\bibliography{References}

\newcommand{\noop}[1]{}
\begin{thebibliography}{DSHNnBW}

\bibitem[AE05]{AEFpure}
Ian~M. Aberbach and Florian Enescu.
\newblock The structure of {F}-pure rings.
\newblock {\em Math. Z.}, 250(4):791--806, 2005.

\bibitem[{\`Al}MBZ12]{AMBZ}
Josep {\`Al}varez~Montaner, Alberto~F. Boix, and Santiago Zarzuela.
\newblock Frobenius and {C}artier algebras of {S}tanley-{R}eisner rings.
\newblock {\em J. Algebra}, 358:162--177, 2012.

\bibitem[BFS13]{BFS}
Ang\'elica Benito, Eleonore Faber, and Karen~E. Smith.
\newblock Measuring singularities with {F}robenius: the basics.
\newblock In {\em Commutative algebra}, pages 57--97. Springer, New York, 2013.

\bibitem[BJNnB19]{BJLdifferentialoperators}
Holger Brenner, Jack Jeffries, and Luis N\'{u}\~{n}ez Betancourt.
\newblock Quantifying singularities with differential operators.
\newblock {\em Adv. Math.}, 358:106843, 89, 2019.

\bibitem[BMS08]{BMMMSK}
Manuel Blickle, Mircea Musta{\c{t}}{\u{a}}, and Karen~E. Smith.
\newblock Discreteness and rationality of {$F$}-thresholds.
\newblock {\em Michigan Math. J.}, 57:43--61, 2008.
\newblock Special volume in honor of Melvin Hochster.

\bibitem[BZ19]{BZ}
Alberto~F. Boix and Santiago Zarzuela.
\newblock Frobenius and {C}artier algebras of {S}tanley-{R}eisner rings ({II}).
\newblock {\em Acta Math. Vietnam.}, 44(3):571--586, 2019.

\bibitem[DS20]{DaoSmirnov}
Hailong Dao and Ilya Smirnov.
\newblock On the generalized {H}ilbert-{K}unz function and multiplicity.
\newblock {\em Israel J. Math.}, 237(1):155--184, 2020.

\bibitem[DSHNnBW]{DSHNBW}
Alessandro De~Stefani, Daniel~J. Hern\'andez, Luis N\'u\~nez Betancourt, and
  Emily~E. Witt.
\newblock $\sigma$-modules and $\sigma$-jumping numbers.
\newblock {\em In preparation}.

\bibitem[DSNnBP18]{DSNBP}
Alessandro De~Stefani, Luis N\'u\~nez Betancourt, and Felipe P\'erez.
\newblock On the existence of {$F$}-thresholds and related limits.
\newblock {\em Trans. Amer. Math. Soc.}, 370(9):6629--6650, 2018.

\bibitem[Fed83]{Fedder'scriterion}
Richard Fedder.
\newblock {$F$}-purity and rational singularity.
\newblock {\em Trans. Amer. Math. Soc.}, 278(2):461--480, 1983.

\bibitem[GW78]{GW1}
Shiro Goto and Keiichi Watanabe.
\newblock On graded rings. {I}.
\newblock {\em J. Math. Soc. Japan}, 30(2):179--213, 1978.

\bibitem[HL02]{CHLG}
Craig Huneke and Graham~J. Leuschke.
\newblock Two theorems about maximal {C}ohen-{M}acaulay modules.
\newblock {\em Math. Ann.}, 324(2):391--404, 2002.

\bibitem[HMTW08]{HMTWF-thresholdsgeneralcase}
Craig Huneke, Mircea Musta{\c{t}}{\u{a}}, Shunsuke Takagi, and Kei-ichi
  Watanabe.
\newblock F-thresholds, tight closure, integral closure, and multiplicity
  bounds.
\newblock {\em Michigan Math. J.}, 57:463--483, 2008.
\newblock Special volume in honor of Melvin Hochster.

\bibitem[Hun00]{HunekeLC}
Craig Huneke.
\newblock The saturation of {F}robenius powers of ideals.
\newblock {\em Comm. Algebra}, 28(12):5563--5572, 2000.
\newblock Special issue in honor of Robin Hartshorne.

\bibitem[HY03]{HaraYoshida}
Nobuo Hara and Ken-Ichi Yoshida.
\newblock A generalization of tight closure and multiplier ideals.
\newblock {\em Trans. Amer. Math. Soc.}, 355(8):3143--3174, 2003.

\bibitem[Kat98]{KatzmanComplexityFrob}
Mordechai Katzman.
\newblock The complexity of {F}robenius powers of ideals.
\newblock {\em J. Algebra}, 203(1):211--225, 1998.

\bibitem[KSSZ14]{KSSZ}
Mordechai Katzman, Karl Schwede, Anurag~K. Singh, and Wenliang Zhang.
\newblock Rings of {F}robenius operators.
\newblock {\em Math. Proc. Cambridge Philos. Soc.}, 157(1):151--167, 2014.

\bibitem[KZ14]{KatzmanZhang}
Mordechai Katzman and Wenliang Zhang.
\newblock Castelnuovo-{M}umford regularity and the discreteness of
  {$F$}-jumping coefficients in graded rings.
\newblock {\em Trans. Amer. Math. Soc.}, 366(7):3519--3533, 2014.

\bibitem[MTW05]{MTWFR}
Mircea Musta{\c{t}}{\u{a}}, Shunsuke Takagi, and Kei-ichi Watanabe.
\newblock F-thresholds and {B}ernstein-{S}ato polynomials.
\newblock In {\em European {C}ongress of {M}athematics}, pages 341--364. Eur.
  Math. Soc., Z\"{u}rich, 2005.

\bibitem[Sch10]{KSCenterofF-purity}
Karl Schwede.
\newblock Centers of {$F$}-purity.
\newblock {\em Math. Z.}, 265(3):687--714, 2010.

\bibitem[SVdB97]{KSVB}
Karen~E. Smith and Michel Van~den Bergh.
\newblock Simplicity of rings of differential operators in prime
  characteristic.
\newblock {\em Proc. London Math. Soc. (3)}, 75(1):32--62, 1997.

\bibitem[Tri18]{TR}
Vijaylaxmi Trivedi.
\newblock Nondiscreteness of {$F$}-thresholds.
\newblock {\em arXiv:1808.07321}, 2018.

\bibitem[Tuc12]{TKFSIG}
Kevin Tucker.
\newblock {$F$}-signature exists.
\newblock {\em Invent. Math.}, 190(3):743--765, 2012.

\bibitem[TW04]{TWFpure}
Shunsuke Takagi and Kei-ichi Watanabe.
\newblock On {F}-pure thresholds.
\newblock {\em J. Algebra}, 282(1):278--297, 2004.

\bibitem[Vra16]{VHK}
Adela Vraciu.
\newblock An observation on generalized {H}ilbert-{K}unz functions.
\newblock {\em Proc. Amer. Math. Soc.}, 144(8):3221--3229, 2016.

\bibitem[Yao06]{YAO}
Yongwei Yao.
\newblock Observations on the {$F$}-signature of local rings of characteristic
  {$p$}.
\newblock {\em J. Algebra}, 299(1):198--218, 2006.

\bibitem[Zha15]{ZhangRegFrob}
Wenliang Zhang.
\newblock A note on the growth of regularity with respect to {F}robenius.
\newblock {\em arXiv:1512.00049}, 2015.

\end{thebibliography}

\end{document}